\documentclass[12pt]{amsart}

\usepackage[margin=1.15in]{geometry}
\usepackage{amscd,amssymb, amsmath,amsfonts, wasysym, mathrsfs, mathtools,enumerate,hhline,color}
\usepackage{graphicx}
\usepackage[all, cmtip]{xy}

\usepackage{url}

\definecolor{hot}{RGB}{65,105,225}

\usepackage[pagebackref=true,colorlinks=true, linkcolor=hot ,  citecolor=hot, urlcolor=hot]{hyperref}

\theoremstyle{plain}
\newtheorem{theorem}{Theorem}[section]

\newtheorem{cor}[theorem]{Corollary}

\newtheorem{lemma}[theorem]{Lemma}

\theoremstyle{definition}

\newtheorem{rmk}[theorem]{Remark}

\newtheorem{ex}[theorem]{Example}
\newtheorem*{ex*}{Example}

\newcommand\sF{{\mathcal F}}

\newcommand\pp{{\mathbb{P}}}

\newcommand\zz{{\mathbb{Z}}}

\newcommand\cc{{\mathbb{C}}}

\newcommand\nn{{\mathbb{N}}}
\def\bH{\mathbb{H}}

\DeclareMathOperator{\codim}{codim}              

\def\bC{\mathbb{C}}

\def\bZ{\mathbb{Z}}

\title[The signed Euler characteristic of very affine varieties]{The signed Euler characteristic of very affine varieties}
\begin{document}
\author{Nero Budur}
\email{Nero.Budur@wis.kuleuven.be}
\address{KU Leuven and University of Notre Dame}
\curraddr{KU Leuven, Department of Mathematics,
Celestijnenlaan 200B, B-3001 Leuven, Belgium}

\author{Botong Wang}
\email{bwang3@nd.edu}
\address{University of Notre Dame}
\curraddr{Department of Mathematics,
 255 Hurley Hall, IN 46556, USA} 
\date{}

\thanks{The first author was partly sponsored by the Simons Foundation and NSA}

\begin{abstract} A conjecture of J. Huh and B. Sturmfels predicts that the sign of the Euler characteristic of a complex very affine variety depends only on the parity of the dimension. The conjecture is true for locally complete intersections. Beyond this case, we construct counterexamples with arbitrarily bad failure. 
\end{abstract}
\maketitle
\section{Introduction}
Let $X$ be a closed irreducible subvariety of $(\cc^*)^n$. In the literature, this is called a very affine variety. When $X$ is a locally complete intersection, $(-1)^{\dim(X)}\chi(X)\geq 0$. This follows from generic vanishing results for perverse sheaves on $(\cc^*)^n$ due to Loeser-Sabbah \cite{LS} (see also Gabber-Loeser \cite{GL}), together with the well-known fact that for a locally complete intersection the shifted sheaf $\bC_X[\dim X]$ is perverse. For the smooth case, see also \cite{Hu}. Since the lci case is not well-known, for the convenience of the reader we include a proof at the end of this article.

It was conjectured by Huh and Sturmfels \cite[page 6]{HS} that the same is true for any closed irreducible subvariety $X$ of $(\cc^*)^n$. In this note, we construct counterexamples by displaying singular surfaces in $(\cc^*)^4$ with arbitrary negative Euler characteristics.

\section{Construction}
We start with a smooth surface $U$ in $(\cc^*)^4$ defined as $$U=\{(w, x, y, z)\in (\cc^*)^4 \mid w+y=x+z=1\},$$ where $w, x, y, z$ are the coordinates in $(\cc^*)^4$. 

We define an action of $\zz/n\zz$ on $(\cc^*)^4$. Let $\xi\in \cc$ be an $n$-th primitive root of unity. We set $\xi (w, x, y, z)=(\xi w, \xi x, \xi^{-1}y, \xi^{-1}z)$. This defines a $\zz/n\zz$ action on $(\cc^*)^4$ by translations. Hence, the quotient $(\cc^*)^4/(\zz/n\zz)$ is again a commutative affine algebraic group. Such algebraic group has to be isomorphic to $(\cc^*)^4$. In fact, we can give an explicit description of the quotient map, which we denote by $p_n: (\cc^*)^4\to (\cc^*)^4$,
$$p_n: (w, x, y, z) \mapsto (w^n, \frac{w}{x}, wy, wz). $$
We denote the image $p_n(U)$ by $U_n$. Then $U_n$ is an irreducible subvariety of $(\cc^*)^4$. 
\begin{theorem}\label{main}
When $n$ is odd, $\chi(U_n)=\frac{3-n}{2}$. 
\end{theorem}
The proof of the theorem will be in the next section. 

\section{Euler characteristic}
We will see that $U_n$ has only isolated singularities, which are analytically equivalent to the transverse intersection of two smooth surfaces in $\cc^4$. Moreover, the normalisation of $U_n$ will be isomorphic to $U$. This allows us to compute the Euler characteristic of $U_n$ by counting the number of singular points on $U_n$. Throughout this section, we assume $n$ is odd. 
\begin{lemma}
For any $1\leq i\leq n-1$, $U\cap \xi^i U$ has exact one point, and the intersection is transverse. Furthermore, $U\cap \xi^i U\cap \xi^j U=\emptyset$ for $1\leq i<j\leq n-1$. 
\end{lemma}
\begin{proof}
Recall that $U=\{(w, x, y, z)\in (\cc^*)^4 \mid w+y=x+z=1\}$. Hence $$\xi^i U=\{(w, x, y, z)\in (\cc^*)^4 \mid \xi^{-i} w+\xi^i y=\xi^{-i} x+\xi^i z=1\}.$$ A direct computation shows $U\cap \xi^i U=\{(\frac{\xi^i}{1+\xi^i}, \frac{\xi^i}{1+\xi^i}, \frac{1}{1+\xi^i}, \frac{1}{1+\xi^i})\}.$ Since the intersection is defined by 4 linear equations, clearly it is transverse. The last part is obvious, since $\frac{\xi^i}{1+\xi^i}\neq \frac{\xi^j}{1+\xi^j}$ for $1\leq i<j\leq n-1$. 
\end{proof}
\begin{cor}
$U_n$ has $\frac{n-1}{2}$ isolated singular points. Moreover, the germ of $U_n$ at any singular point is analytically equivalent to the germ of $\{w=x=0\}\cup \{y=z=0\}$ in $\cc^4$ at origin. In other words, locally the singularity is obtained by the transverse intersection of two smooth surfaces. 
\end{cor}
\begin{proof}
The quotient map $p_n: (\cc^*)^4\to (\cc^*)^4$ restricts to a map $q_n: U\to U_n$. The map $q_n$ is an isomorphism on $U-\bigcup_{1\leq i\leq n-1}(U\cap \xi^i U)$. Every intersection $U\cap \xi^i U$ will create a singular point on $U_n$, which is the image of $U\cap \xi^i U$ under $q_n$. Since the intersection $U\cap \xi^i U$ is transverse, and since this intersection is not contained in any other $\xi^j U$, the corresponding singular point in $U_n$ is locally isomorphic to the transverse intersection of two smooth surfaces. Notice that $\xi^{-i} (U\cap \xi^i U)=U\cap \xi^{n-i} U$. $U\cap \xi^i U$ and $U\cap \xi^{n-i} U$ give the same singular point in $U_n$. On the other hand, each singular point in $U_n$ comes from exactly two such intersections. Therefore, $U_n$ has exactly $\frac{n-1}{2}$ singular points. 
\end{proof}
\begin{proof}[Proof of Theorem \ref{main}]
The map $\cc^2\to \cc^4$, defined by $(a, b)\mapsto (a, b, 1-a, 1-b)$ induces an isomorphism between $(\cc-\{0, 1\})\times (\cc-\{0, 1\})$ and $U$. Therefore, by K\"{u}nneth's formula $\chi(U)=1$. According to the corollary, $U_n$ is obtained from $U$ by attaching $\frac{n-1}{2}$ pairs of points. Thus $\chi(U_n)=1-\frac{n-1}{2}=\frac{3-n}{2}$. 
\end{proof}

\begin{ex}The smallest $n$ for which $U_n$ is a counterexample to the above-mentioned conjecture is $n=5$: $\chi(U_5)=-1$. Using a computer algebra program we can compute the equations for $U_5$. This surface is the common zero locus in $(\bC^*)^4$ of the following 4 equations in 4 variables:
$$
\begin{aligned}
0 &= t_2^2t_4^2+t_2^2t_3-t_2^2t_4-2t_2t_3t_4-{t_2}{t_3}+{t_3}^2+{t_2}{t_4},\\
0 &={t_2}{t_3}^2{t_4}+2{t_2}{t_3}^2-{t_3}^3-3{t_2}{t_3}{t_4}-{t_1}{t_2}-{t_2}{t_3}+{t_3}^2+{t_2}{t_4}+{t_1},\\
0 &=  {t_3}^4-{t_3}^3{t_4}+{t_1}{t_2}{t_4}^2-5{t_2}{t_3}{t_4}^2+{t_1}{t_2}{t_3}-5{t_2}{t_3}^2+3{t_3}^3-{t_1}{t_2}{t_4}-{t_1}{t_3}{t_4}+5{t_2}{t_3}{t_4}+\\
& \quad+2{t_3}^2{t_4} +2{t_2}{t_4}^2-3{t_1}{t_3}+2{t_2}{t_3}-2{t_3}^2+3{t_1}{t_4}-2{t_2}{t_4}, \\   
0 &={t_1}{t_2}{t_3}{t_4}^2-{t_3}^3{t_4}^2+{t_1}{t_2}{t_4}^3-5{t_2}{t_3}{t_4}^3+{t_1}{t_2}{t_3}^2-{t_1}{t_3}^2{t_4}+3{t_3}^3{t_4}-4{t_1}{t_2}{t_4}^2\\
& \quad -{t_1}{t_3}{t_4}^2+7{t_2}{t_3}{t_4}^2+2{t_3}^2{t_4}^2+2{t_2}{t_4}^3-3{t_1}{t_2}{t_3}+2{t_1}{t_3}^2+12{t_2}{t_3}^2-6{t_3}^3+3{t_1}{t_2}{t_4}+\\
& \quad+3{t_1}{t_3}{t_4} -15{t_2}{t_3}{t_4}-3{t_3}^2{t_4}+3{t_1}{t_4}^2-3{t_2}{t_4}^2-{t_1}^2-5{t_1}{t_2}-{t_1}{t_3}-6{t_2}{t_3}+6{t_3}^2-\\
& \quad-4{t_1}{t_4}+6{t_2}{t_4}+6{t_1}.
\end{aligned}
$$
(Note: an output via the command {\tt mingens} in Macaulay2 contains one extra, redundant, equation due to being used in an inhomogeneous situation. We thank a referee for kindly point this out.)

\end{ex}

\begin{rmk} For the motivation behind the conjecture on Euler characteristics and relations with maximum likelikhood degrees, we refer to the survey \cite{HS}. Our examples leave open \cite[Conjecture 1.8]{HS} on maximum likelihood degrees.
\end{rmk}

\section{Locally complete intersections}

For the conveniece of the reader, we give the proof of the following result due to \cite{LS, GL}:

\begin{theorem}
Let $X$ be a closed subvariety of $(\bC^*)^n$ of dimension $d$ such that the shifted complex $\bC_X[d]$ is a perverse sheaf on $(\bC^*)^n$. Then $(-1)^{d}\chi (X)\ge 0$. In particular, this holds for $X$ with at most locally complete intersection singularities.
\end{theorem}
\begin{proof} The space of rank-one local systems on $Y=(\bC^*)^n$ is the same as the space of characters $Char(Y)=Hom (H_1(Y,\bZ),\bC^*)\cong(\bC^*)^n$.
Consider the cohomology jump loci of rank-one local systems on $Y$ relative to a complex $\sF$ of constructible sheaves (in the analytic topology):
$$
V^i_k(\sF)=\{ L  \in Char (Y)\mid \dim_\bC \bH^i(Y,\sF\otimes_\bC L)\ge k \},
$$
where $\bH$ denotes hypercohomolgy. These are closed subschemes of $Char (Y)$. If $\sF$ is perverse, then $$\codim V^i_k(\sF)\ge i$$ for $k\ge 1$, by a fundamental result of \cite{LS, GL}. In particular,  $\bH^i(Y, \sF\otimes L)=0$ for a general $L$ and $i\ne 0$. Let now $\sF=\bC_X[d]$ viewed as a perverse sheaf on $Y$. Then, for a general $L$, $H^{i+d}(X, L_{|X})=\bH^{i}(Y,\bC_X[d]\otimes L)=0$ for $i\ne 0$. Hence
$$
(-1)^{d}\chi (X) =(-1)^{d}\chi (X, L_{|X})=\dim H^{d}(X, L_{|X})\ge 0,
$$
the first equality being true for any rank-one local system on a complex variety. It is well-known that shifted complex $\bC_X[d]$ is perverse on $X$ if $X$ has at most lci singularities \cite[Theorem 5.1.20]{Di}, and that it remains perverse when viewed on $Y$ via the direct image under the embedding of $X$ in $Y$.
\end{proof}

\bigskip
\end{document}